\documentclass[12pt,a4paper]{amsart}

\usepackage{amssymb}

\usepackage{graphicx,amssymb,amsfonts,epsfig,amsthm,a4,amsmath,url, enumerate}
\usepackage[latin1]{inputenc}
\usepackage{hyperref}

\usepackage[alphabetic]{amsrefs} 
\usepackage[usenames,dvipsnames]{color}

\newtheorem{thm}{Theorem}[section]
\newtheorem{cor}[thm]{Corollary}
\newtheorem{lem}[thm]{Lemma}

\newtheorem{prop}[thm]{Proposition}

\theoremstyle{definition}

\newtheorem{que}[thm]{Question}

\newtheorem{rem}[thm]{Remark}

\numberwithin{equation}{section}

\newcommand{\N}{\mathbf{N}}
\newcommand{\Z}{\mathbf{Z}}
\newcommand{\R}{\mathbf{R}}
\newcommand{\CC}{\mathbf{C}}
\newcommand{\Q}{\mathbf{Q}}

\newcommand{\GL}{\textnormal{GL}}

\newcommand{\Ker}{\text{Ker}}

\newcommand{\barwr}{\;\bar{\wr}\; }

\begin{document}
\title{On embeddings into compactly generated groups}

\author{Pierre-Emmanuel Caprace}
\address{Universit\'e catholique de Louvain, IRMP, Chemin du Cyclotron 2, 1348 Louvain-la-Neuve, Belgium}
\email{pe.caprace@uclouvain.be}
\thanks{P-E.\ C. was partly supported by FNRS grant F.4520.11 and the European Research Council}

\author{Yves Cornulier}%
\address{Laboratoire de Math\'ematiques\\
B\^atiment 425, Universit\'e Paris-Sud 11\\
91405 Orsay\\FRANCE}
\email{yves.cornulier@math.u-psud.fr}

%



\date{November 22, 2012}

\begin{abstract}
We prove that there is a second countable locally compact group that does not embed as a closed subgroup in any compactly generated locally compact group, and discuss various related embedding and non-embedding results.
\end{abstract}

\maketitle

\section{Introduction}
The Higman--Neumann--Neumann Theorem \cite{HNN} ensures that every countable group embeds as a subgroup of a finitely generated group, indeed 2-generated (relying on a fundamental construction referred to since then as \emph{HNN-extension}). This was a major breakthrough, providing some of the first evidence that finitely generated groups are not structurally simpler than countable groups and thus are far from tame or classifiable. Later B.~and H.~Neumann~ \cite{NN}   gave an alternative construction, showing for instance that every countable $k$-solvable group (i.e. solvable of derived length at most $k$) embeds as a subgroup of a finitely generated $(k+2)$-solvable group. Further refinements by Ph.~Hall~\cite{Hall} and P.~Schupp~\cite{Schupp} in a slightly different direction showed that every countable group embeds in a $2$-generated \emph{simple} group. 

In the present paper, we address similar questions in the context of \textbf{locally compact topological groups}, which will be abbreviated henceforth by the term \textbf{l.c.\ groups}.
Recall that locally compact groups are a natural generalization of discrete groups, the counterpart of \emph{countability}, respectively \emph{finite generation}, being \emph{$\sigma$-compactness}, resp.~\emph{compact generation}. 
 A prototypical example of an embedding of a non-compactly generated l.c.~group into a compactly generated one is the embedding of the $p$-adic additive group $\Q_p$ into the affine group $\Q_p\rtimes\Q_p^\times$ (or its discrete cousin, the embedding of the additive underlying group of the ring $\Z[1/p]$ into the Baumslag-Solitar group $\Z[1/p]\rtimes_p\Z$). 

It is natural to ask whether analogues of the HNN Theorem hold in the context of l.c. groups. In that context, an \textbf{embedding} $\varphi \colon H \to G$ of an l.c.~group $H$ to an l.c.~group $G$ is defined as a continuous injective homomorphism (with potentially non-closed image). In the non-discrete setting, several natural  variants of the question can be considered: 
\begin{itemize}
\item Given a $\sigma$-compact l.c. group $H$, is there any embedding $\varphi \colon H \to G$ into a compactly generated l.c.~group $G$? 
\item Is there one with closed image? 
\item Is there one with open image?
\end{itemize}
 It turns out that, whenever the topology on $H$ is non-discrete, the answers to these questions are not always positive, and depend heavily on the algebraic structure of $H$. The results of this note  are intended to illustrate that matter of fact. We start with a positive result in case the algebraic structure of $H$ is the simplest possible, namely $H$ is abelian.

\begin{thm}\label{t:embab}
Every $\sigma$-compact {abelian} l.c.~group $A$ embeds as an open subgroup of a compactly generated group $G$, which can be chosen to be 3-solvable. If moreover $A$ is totally disconnected, respectively second countable, resp.\ both, then $G$ can also be chosen to enjoy the same additional properties.
\end{thm}

In particular, the additive group of Adeles, defined as a restricted product of all $\Q_p$ (see Section \ref{adels} for the definition) is isomorphic to an open subgroup of a compactly generated locally compact group. In contrast, the Adeles are used to prove the following result, which shows {in particular} that Theorem~\ref{t:embab} cannot be generalized to solvable groups.

\begin{thm}\label{t:noemb}
There exists a $\sigma$-compact \emph{metabelian} l.c.~group $M$ not isomorphic to any closed subgroup of any compactly generated l.c.~group. 

Moreover $M$ can be chosen to be second countable and  totally disconnected.
\end{thm}

The proof of Theorem \ref{t:noemb} is based on the now classical observation, due to H.~Abels \cite[Beispiel~5.2]{Abels}, that every compactly generated l.c.~group admits, in a somewhat natural way, a continuous proper action on a connected graph of bounded degree (see Proposition~\ref{prop:lcg} below). Using similar  ideas, we obtain the following result, which shows that the HNN Theorem fails in the non-discrete setting, even if one allows embeddings with potentially non-closed images.

\begin{thm}\label{t:simple}
There exists a second countable (hence $\sigma$-compact), topologically simple totally disconnected l.c.~group $S$, such that every continuous (or even abstract) homomorphism of $S$ to  {any} compactly generated l.c.~group is trivial.
\end{thm}

This result suggests that, as opposed to the discrete case, compact generation for non-discrete l.c.~groups imposes a strong constraint, making it thus plausible that some of the pathologies that are unavoidable in the general case do not occur. On the other hand, the local compactness is absolutely  essential since  it is known \cite{Pes} that every $\sigma$-compact topological Hausdorff group is isomorphic to a closed subgroup of some compactly generated topological Hausdorff group.

\medskip 
We finally present a result illustrating the difference between embeddings with closed and open images.

\begin{thm}\label{opcl}
There exists a second countable (hence $\sigma$-compact) l.c.~group $H$ which is isomorphic to a {\em closed} subgroup of a compactly generated l.c.~group, but not to any {\em open} subgroup of any compactly generated l.c.~group. 

Moreover, $H$ can be chosen to be of the form $K\rtimes\Gamma$ with $\Gamma$ discrete abelian, and $K$ compact abelian, either connected or profinite. It can also be chosen to be a Lie group.   
\end{thm}

One part of the implication in Theorem~\ref{opcl} is the following general fact, which is based  on a wreath product construction.

\begin{prop}\label{p:compact-by-disc}
Any compact-by-\{countable discrete\} l.c.~group embeds as a closed subgroup in a compact-by-\{finitely generated discrete\} l.c.~group.
\end{prop}

{Similarly as in Theorem~\ref{t:embab},} this proposition illustrates that embedding theorems can hold in the non-discrete case when the algebraic or topological structure of the group $H$ is not too complicated. 

\bigskip
We finish by mentioning some related natural questions which we have not been able to answer.

\begin{que}Is every second countable (real) Lie group isomorphic to a closed subgroup of a compactly generated locally compact group? of a compactly generated Lie group? Same question for $p$-adic Lie groups.
\end{que}
The answer to the latter questions, with `closed subgroup' replaced by `open subgroup', is negative for both real and $p$-adic Lie groups, see the examples in Section \ref{qha}.

\medskip

\noindent {\bf Acknowledgments.} We thank Pierre de la Harpe for discussions and in particular for pointing out the reference \cite{Pes}.

\section{Locally compact groups, Lie groups and locally finite graphs}

We shall use the following general fact on l.c.~groups; the first part follows from the solution to Hilbert's fifth problem, the second is an elementary but crucial observation due to H.~Abels. 

\begin{prop}\label{prop:lcg}
Let $G$ be an l.c.~group and $V$ be any identity neighbourhood. 

\begin{enumerate}[(i)]
\item{\rm (Yamabe)} If $G$ is connected-by-compact (i.e. if $G/G^\circ$ is compact), then $V$ contains a compact normal subgroup $K$ of $G$ such that $G/K$ is a connected Lie group. 

\item{\rm (Abels)} If $G$ is totally disconnected and compactly generated, then $V$ contains a compact normal subgroup $W$ of $G$ such that $G/W$ admits a faithful, continuous proper vertex-transitive action on some connected locally finite graph.
\end{enumerate}
\end{prop}

\begin{proof}
For (i), see \cite{MZ}*{Th.~IV.4.6}. For (ii), originally observed in \cite[Beispiel~5.2]{Abels}, we refer to \cite{Mon}*{\S 11.3}. 
\end{proof}

We deduce the following useful criterion for the non-existence of embeddings into compactly generated l.c.~groups. 

\begin{prop}\label{homtri}
Let $H$ be an l.c.\ group. The following are equivalent
\begin{enumerate}
\item Every continuous homomorphism of $H$ to a compactly generated l.c.~group is trivial
\item\label{c2} The following two conditions are satisfied:
\begin{enumerate}
\item\label{tota} every continuous homomorphism of $H$ to a compactly generated totally disconnected l.c.~group is trivial;
\item\label{nlr} every continuous linear representation $H\to\GL_n(\CC)$ is trivial for all~$n$.
\end{enumerate}
\end{enumerate}
Moreover, a sufficient condition for (\ref{tota}) is that $H$ has no nontrivial continuous action on any connected graph of bounded degree.
\end{prop}
\begin{proof}
One implication is trivial. Assume that (\ref{c2}) holds. Let $G$ be a compactly generated l.c.\ group and $f \colon H \to G$ be a continuous homomorphism. Considering the composite map $H\to G\to G/G^\circ$ and in view of (\ref{tota}), we see that $f(H)\subset G^\circ$. If $f$ is not the trivial map, some identity neighbourhood $V$   in $G$ does not contain the image of $f$. 
By Proposition~\ref{prop:lcg}(i), there is a compact normal subgroup $K$ of $G^\circ$ contained in $V$ such that $L=G^\circ/K$ is a (connected) Lie group. So the composite map $H\to L$ is non-trivial. Using the adjoint representation of $L$ and (\ref{nlr}), we see that it maps $H$ into the center of $L$. On the other hand, it follows from Pontryagin duality and (\ref{nlr}) that $H$ admits no nontrivial continuous homomorphism to any abelian l.c.\ group. So we get a contradiction, and thus $f$ is the trivial homomorphism.

Let us now assume that $H$ has no nontrivial continuous action on any connected graph of bounded degree   and let us check that (\ref{tota}) holds. Let $f\colon H\to G$ be a continuous homomorphism, with $G$ a compactly generated, totally disconnected l.c.~group. If $f$ is non-trivial,  some identity neighbourhood $V$   in $G$ does not contain the image of $f$. By Proposition~\ref{prop:lcg}(ii), there is a compact normal subgroup $K$ of $G$ contained in $V$ such that $G/K$ acts continuously, faithfully, vertex-transitively on a connected locally finite graph. The hypothesis made on $H$ implies that the restriction of this action to $H$ is trivial. Thus $f(H)$ is contained in $K$, hence in $V$, which is a contradiction.\end{proof}

\begin{rem}
We do not know if, conversely, (\ref{tota}) implies that $H$ has no nontrivial continuous action on any connected graph of bounded degree. In other words, does the existence of a continuous non-trivial action on a connected graph of bounded degree imply the existence of such an action on a vertex-transitive graph? The same question, replacing ``non-trivial" by ``proper", can also naturally be addressed.
\end{rem}

Finally, we record an elementary fact, allowing us in suitable situations to exclude actions on some connected locally finite graphs. 

\begin{lem}\label{lem:graphs}
Let $G$ be an l.c.~group acting continuously by automorphisms on a connected graph all of whose vertices have degree~$\leq d$. Then every vertex stabiliser $O$ is open in $G$ and, for any prime $p>d$, every closed pro-$p$ subgroup of $O$  acts trivially on the graph. In particular, if $G$ admits an open pro-$p$-group, then the action has an open kernel.
\end{lem}

\begin{proof}
Let $O$ be a vertex stabiliser, which is open in $G$ since the action on the graph is assumed continuous. Given any closed subgroup $H$ of $O$ which acts non-trivially on the graph, there is a vertex $v$ fixed by $H$ and adjacent to some vertex which is not fixed by $H$. In particular $H$  admits some non-trivial continuous permutation action on the set of neighbours of $v$, which is  a set of at most $d$ elements. It follows that $H$ cannot be pro-$p$ for any $p>d$. 
\end{proof}

\section{Proof of Theorem \ref{t:embab}}

Recall that any totally disconnected l.c.~group contains compact open subgroups.  {Moreover every abelian l.c.~group $A$ has a (non-canonical) decomposition as a topological direct product $\R^n\times W$, where $W$ is compact-by-discrete,} and the discrete quotient is countable as soon as $A$ is $\sigma$-compact. Those facts could be used to deduce (a part of) Theorem~\ref{t:embab} from Proposition~\ref{p:compact-by-disc}. This is however not what we shall do here, and present rather a simpler direct argument.

\medskip
We begin with an easy classical result.

\begin{lem}\label{Zgamma}
There exists a finitely generated group $\Gamma$ whose center contains a free abelian group $Z$ of countable rank; $\Gamma$ can be chosen to be 3-solvable.
\end{lem}
\begin{proof}
If $t$ is an indeterminate, the reader can check that the three matrices
$$\begin{pmatrix} 1 & 0 & 0\\ 0 & t & 0\\ 0 & 0 & 1\end{pmatrix},\;\begin{pmatrix} 1 & 1 & 0\\ 0 & 1 & 0\\ 0 & 0 & 1\end{pmatrix},\;\begin{pmatrix} 1 & 0 & 0\\ 0 & 1 & 1\\ 0 & 0 & 1\end{pmatrix}$$
generates a group containing the set of all matrices 
of the form $$\begin{pmatrix} 1 & 0 & P(t)\\ 0 & 1 & 0\\ 0 & 0 & 1\end{pmatrix},\;P(t)\in\Z[t,1/t] $$
as a central,  {infinitely generated subgroup. (This construction is due to P.\ Hall \cite{Hall}*{Theorem~7}.)}
\end{proof}

\begin{lem}\label{cocse}
If $G$ is a $\sigma$-compact l.c.~group, then it has a cocompact closed separable subgroup.
\end{lem}
\begin{proof}
By the Kakutani--Kodaira Theorem, there is a compact normal subgroup $K$ such that $G/K$ is second countable. So $G/K$ admits a dense countable subset. Lift this subset to $G$, and let $D$ be the abstract (countable) group it generates. So $G = \overline{KD} = K\overline{D}$ since $K$ is compact. Thus $\overline D$ is cocompact; moreover it is separable by construction.
\end{proof}

\begin{proof}[Proof of Theorem \ref{t:embab}]
By Lemma \ref{cocse}, there is a cocompact closed separable subgroup in $A$. In other words, there is a homomorphism $f \colon Z\to A$ whose image has cocompact closure, where $Z=\Z^{(\omega)}$ is the restricted product of countably many copies of the infinite cyclic group. In view of Lemma~\ref{Zgamma}, the group $Z$ can be embedded as a central subgroup of a finitely generated group $\Gamma$ (which can be chosen to be 3-solvable). The graph $F$ of $f$ is a closed discrete central subgroup of $\Gamma\times A$. Since $f$ is injective, it follows that the mapping of $A$ into $G=(\Gamma\times A)/F$ is injective. 
Moreover $A$ has  open image (because the quotient map is open). So $A$ lies as an open (and central) subgroup of $G$. The latter group is compactly generated: indeed, the closure of the subgroup generated by a finite generating subset of $\Gamma$ is cocompact.
By construction, if $A$ is second countable, resp.\ totally disconnected, then so is $G$. 
\end{proof}

\section{Proof of Theorem \ref{t:noemb}}\label{adels}

Consider $B_p=\Q_p\rtimes_p\Z$, where the notation $\rtimes_p$ means that the $\Z$-action is through multiplication by powers of $p$. Let $A$ be the group of Adeles, namely the set of elements in $\prod_p\Q_p$ ($p$ ranging over all primes) whose projection in $\prod\Q_p/\Z_p$ is finitely supported, endowed with the ring topology for which $\prod\Z_p$ is a compact open subring. In the product $\prod_p B_p=(\prod_p\Q_p)\rtimes\prod_p\Z$, consider the subgroup $Z=\bigoplus_p\Z$, and endow it with the discrete topology. Finally define $M=A\rtimes Z\subset\prod_p B_p$. The group $M$ is metabelian admits a unique Hausdorff group topology for which $\prod_p\Z_p$ is a  compact open subgroup. In particular $M$ is locally compact.

Theorem~\ref{t:noemb} is a consequence of the following.

\begin{prop}\label{mainprop}
There is no embedding of $  M$ as a closed subgroup of any compactly generated l.c.~group. 

More precisely, given any continuous   homomorphism $f \colon M \to G$ to a compactly generated l.c.~group $G$,  there exists $p_0$ such that $\overline{f(\Q_p)}$ is a compact connected group  for all $p\ge p_0$.
\end{prop}

We begin by two lemmas on homomorphisms from $\Q_p$ into locally compact groups.

\begin{lem}\label{qpcc}
For every continuous homomorphism $f \colon \Q_p \to G$ of $\Q_p$ to a connected-by-compact l.c.~group $G$, the closure of the image $\overline{f(\Q_p)}$ is compact and connected.
\end{lem}
\begin{proof}
Assume first that $G$ is a virtually connected Lie group. Since $\Q_p$ is divisible, it has no nontrivial finite quotient. Thus $\overline{f(\Q_p)}$ is a closed abelian subgroup of a connected Lie group, so is isomorphic to a product $\Gamma\times \R^k\times T$ for some finitely generated abelian group $\Gamma$ and torus $T$. Invoking again that $\Q_p$ has no nontrivial finite quotient, we find $\Gamma=\{0\}$.  Since $\R^k\times T$ has no small subgroup, the kernel of $f$ must be open in $\Q_p$.  In particular $f(\Q_p)$ is a torsion group, from which we infer that $k=0$. Therefore $\overline{f(\Q_p)}$ is a torus; in particular it is  compact.

Coming back to the general case, we now let $W$ be the maximal compact normal subgroup of $G$, which exists by Proposition~\ref{prop:lcg}(i). Proposition \ref{mainprop} ensures that  $G/W$ is a virtually connected Lie group.
By the special case above, we deduce that, denoting $K=\overline{f(\Q_p)}$, we have $KW/W$ is compact. Hence $K$ is compact as well.
Since $K/K^\circ$ is profinite and $\Q_p$ has no nontrivial finite quotient, $K=K^\circ$, i.e.\ $K$ is connected.
\end{proof}

\begin{rem}
It follows from Pontryagin duality that $\Q_p$ has a continuous homomorphism with dense image into the circle, and also has an injective continuous homomorphism with dense image into the Pontryagin dual $\hat{\Q}$ of $\Q$, which is a connected compact group.
\end{rem}

\begin{lem}\label{qptd}
Every nontrivial continuous homomorphism $f \colon \Q_p \to G$ of $\Q_p$ to a totally disconnected l.c.~group $G$ is proper, and has either a compact open kernel or is an isomorphism to its (closed) image.
\end{lem}
\begin{proof}
We can suppose that $f$ has dense image, so $G$ is abelian. Let $U$ be a compact open subgroup in $G$. Then $f^{-1}(U)$ is an open subgroup of $\Q_p$. If it is all of $\Q_p$, then $U=G$ and since $U$ is profinite and $\Q_p$ has no nontrivial finite quotient, it follows that $U=\{1\}$. Otherwise, $f^{-1}(U)$ is a compact  open subgroup, so $f$ is proper and in particular has closed image and is the quotient map by some compact subgroup, giving the two possibilities.
\end{proof}

\begin{lem}\label{imap2}
Every continuous homomorphism $f \colon B_p \to G$ of the group $B_p =\Q_p\rtimes_p\Z$ to a totally disconnected l.c.~group $G$ satisfies the following alternative: either $f$ is a topological isomorphism to its closed image, or $f(\Q_p)$ is trivial. 
\end{lem}
\begin{proof}
If $f(\Q_p)$ is nontrivial, then $f$ is proper in restriction to $\Q_p$ {by Lemma \ref{qptd}}. Since the only compact subgroup of $\Q_p$ that is normal in $B_p$ is the trivial group, it follows from Lemma~\ref{qptd} that the restriction of $f$ to $\Q_p$ is an isomorphism to its closed image.

Let $\Omega$ be the normalizer of $f(\Q_p)$ in $G$; this is a closed subgroup and there is a unique continuous homomorphism $\rho:\Omega\to\Z$ such that the conjugation by $g\in\Omega$ on $f(\Q_p)$ multiplies the Haar measure by $p^{\rho(g)}$. In restriction to $\Z$, we see that $\rho\circ f$ is the identity. It follows that $f(\Q_p)$ is open in $f(B_p)$ and that $f$ is proper.
\end{proof}

\begin{proof}[Proof of Proposition \ref{mainprop}]
Let $f \colon M \to G$ be an arbitrary continuous homomorphism to a compactly generated l.c.~group $G$.
Note that  $G/G^\circ$ is a compactly generated totally disconnected l.c.~group. Therefore it has continuous proper action on a connected graph of degree $d$, for some $d$, by  Proposition~\ref{prop:lcg}(ii).  By Lemma \ref{lem:graphs}, for every $p>d$, the restriction to $\Q_p$ of the $G$-action on this graph has an open kernel. Hence, by Lemma~\ref{imap2}, the action of $\Q_p$ on this graph is trivial for all $p>d$. 

Let $W/G^\circ$ is the (compact) kernel of the $G$-action on the graph. Thus $W$ is connected-by-compact, and contains $f(\Q_p)$ for all $p>d$. In view of Lemma~\ref{qpcc}, this implies that for all $p>d$, the group $\overline{f(\Q_p)}$ is compact and connected. 
\end{proof}

\begin{rem}
Proposition \ref{mainprop} and Lemma \ref{qptd} together show that there is no injective continuous homomorphism from $M$ to any {\em totally disconnected} compactly generated l.c.\ group. On the other hand, it admits an injective continuous homomorphism (not proper!) to a compactly generated l.c.\ group, which can be obtained as follows: start from the dense embedding $\Q\subset\Q_p$; it induces a dense embedding $\Q_p\subset\hat{\Q}$, where $\hat{\Q}$ is the Pontryagin dual of $\Q$ (this is a compact connected group). The multiplication by $p$ is an automorphism of $\Q$ and thus induces a topological automorphism of $\hat{\Q}$, also given by multiplication by $p$. So we obtain a continuous injective homomorphism $M\to \prod_p\hat{\Q}\rtimes Z$, where the $p$th component of $Z$ acts on the $p$th component of the compact group $\prod_p\hat{\Q}$ by multiplication by $p$. By Theorem \ref{t:embab}, the latter group embeds into a compactly generated l.c.\ group.
\end{rem}

\section{Some groups of permutations}

\subsection{A non-embedding criterion}
\begin{prop}\label{noho}
Let $H$ be a topologically simple totally disconnected locally compact group. Assume that $H$ has a compact open subgroup $K$ such that for every $k$, the group $K$ possesses, for some prime $p>k$, a closed subgroup topologically isomorphic to a non-trivial pro-$p$-group (e.g., $K$ has some element of order $p$). Then $H$ admits no nontrivial continuous homomorphism into any compactly generated locally compact group. 
\end{prop}
\begin{proof}
We use the criteria from Proposition~\ref{homtri} applied to $H$, in which we can replace ``nontrivial" by ``faithful" since $H$ is topologically simple. Thus we only have to show:
\begin{enumerate}
\item\label{snc1} $H$ has no faithful continuous action on any connected graph of bounded degree;
\item\label{snr1} $H$ has no faithful continuous representation into $\GL_n(\CC)$ for any $n$.
\end{enumerate}
The condition (\ref{snr1}) is immediate as $H$ has small nontrivial subgroups whereas $\GL_n(\CC)$ has none. 

Now consider a continuous action of $H$ on a connected graph of bounded degree, say $\le d$. Fix a vertex $x_0$. Then the stabiliser $K_{x_0}$ of $x_0$ in $K$ is open, hence of finite index in $K$. Therefore, the hypothesis implies that $K_{x_0}$, and hence also the full stabiliser $H_{x_0}$, contains a non-trivial pro-$p$-subgroup $L$ for some prime $p>d$. On the other hand, Lemma~\ref{lem:graphs} implies that $L$ acts trivially on the graph, so (\ref{snc1}) holds.
\end{proof}

\subsection{Proof of Theorem~\ref{t:simple}} We here prove the continuous case of Theorem~\ref{t:simple}. The case of abstract homomorphisms is postponed to \S\ref{absh}.

There exist various sources of topologically simple groups satisfying the criterion of Proposition~\ref{noho} and, hence, the conclusions of Theorem~\ref{t:simple}. We shall content ourselves with describing one of them, following a construction due to Akin, Glasner and Weiss in \cite[\S 4]{AGW}; let us point out that those examples were independently obtained as part of a more general construction by Willis \cite[\S 3]{Wil}.  

The construction goes as follows. Fix a sequence $u=(u_k)_{k\ge 0}$ of integers greater than 2. Define the graph $\mathcal{G}=\mathcal{G}(u)$  (non-oriented, without self-loops) as a disjoint union of complete graphs $\mathcal{G}_k$ on $u_k$ elements, whose vertex set is also denoted by $\mathcal{G}(u)$. Let us call the \textbf{height function} $h$ the function $\mathcal{G}\to\N$ mapping any $v\in\mathcal{G}_k$ to $k$. Note that $h$ completely characterizes the graph structure.

Given a self-map $f:\mathcal{G}\to\mathcal{G}$, we call a vertex $v\in\mathcal{G}_u$ {\bf singular} if $h(f(v))\neq v$. We call the self-map $f$ {\bf almost regular} if only finitely many vertices are singular. If $f$ is a permutation, we say that $f$ is an {\bf almost automorphism} of the graph with height function $(\mathcal{G},h)$ if both $f$ and $f^{-1}$ are almost regular. The group of almost automorphisms of $(\mathcal{G},h)$ is denoted by~$S$ (or $S(u)$ if we need specify it). Its subgroup of automorphisms of $(\mathcal{G},h)$, i.e.\ those $f$ preserving the height and the graph structure, is denoted by $K$ (or $K(u)$).

Note that $K$ is naturally isomorphic to the product $\prod_{k=0}^\infty\textnormal{Sym}(u_k)$, which makes it a compact group. 
The group $S$ is endowed with the unique left-invariant topology making $K$ a compact open subgroup; this topology is obviously locally compact and is a group topology, as checked in \cite{AGW}*{\S 4}. It is the union of an increasing union of compact subgroups $(K_n)_{n\ge 0}$, where $K_n$ is defined as those elements in $S$ all of whose singularities and pairs of singularities lie in $\bigcup_{i\le n}\mathcal{G}_i$; note that $K_0=K$ and that $K_n$ is topologically isomorphic to 
\[\textnormal{Sym}(u_{0}+\dots +u_n)\times\prod_{k\ge n+1}\textnormal{Sym}(u_k).\] 

Define $K_n^+$ as its closed subgroup 
\[\textnormal{Alt}(u_{0}+\dots +u_n)\times\prod_{k\ge n+1}\textnormal{Alt}(u_k).\] 
Note that the sequence $(K_n^+)$ is increasing; we define $S^+$, as an abstract group, as the union $\bigcup_{n\ge 0}K_n^+$. Endow it with the left-invariant topology making $K_0^+$ a compact open subgroup. For the same reason as $S$, this is a group topology.

Finally, we define $A<S$ and $A^+< S^+$ as the subgroups consisting of the \textbf{finitary} permutations, i.e. the permutations with finite support. Clearly $A$ is the group of all finitary permutations on the vertices of $\mathcal G$, while $A^+$ is the index two subgroup of $A$ consisting of the alternating finitary permutations. 

\begin{rem}
It is easily seen that   $A^+$  (resp.\ $A$) is dense as a subgroup of  $S^+$ (resp.\ $A$). Moreover $A^+$ is also dense in $S$: indeed, since $A$ is dense, it is enough to show that any transposition $(x\;y)$ in $S$ can be approximated by a sequence of elements of $A^+$. This is indeed the case, using a sequence of double transpositions $(x\;y)(x_k\;y_k)$ with $x_k,y_k$ distinct elements of the same height tending to infinity with $k$.

This implies in particular that the   embedding of $S^+$ into $S$, which is continuous, is not closed: indeed, its image is a proper subgroup which is dense since it contains $A^+$.
 \end{rem}

\begin{rem}In \cite{AGW}, it is shown that $S$ has a dense conjugacy class, under the assumption that $\lim u_k=\infty$. The precise statement of \cite{AGW}*{Theorem 4.4} actually shows that such a conjugacy class can be found inside $S^+$, and also shows that $S^+$ itself admits  a dense conjugacy class.
\end{rem}

Let us show the following related but independent result.

\begin{prop}\label{gsi}
Every non-trivial normal subgroup of $S^+$ (resp.\ $S$) contains $A^+$, and is thus dense. In particular $S^+$ and $S$ are both topologically simple.
\end{prop}
Note that $S^+$ and $S$ are not abstractly simple, since $A^+$ is a proper dense normal subgroup in both.
\begin{proof}
Let $s$ be a nontrivial element in $S^+$ (resp.\ $S$) and $t\in A^+$. Let $N$ be the normal subgroup generated by $s$. For some $n\ge 2$ (this ensures $u_{0}+\dots+u_n\ge 5$), the element $s$ belongs to $K_n$ and $t$ has support in the finite set $X=\bigcup_{i=0}^n\mathcal{G}_i$ and can thus be viewed as an element of $\textnormal{Alt}(X)$. The commutator $s'$ of $s$ and a suitable element of $\textnormal{Alt}(X)$ is a nontrivial element of $N\cap\textnormal{Alt}(X)$. By simplicity of $\textnormal{Alt}(X)$, it follows that $t\in N$.
\end{proof}
We deduce the following corollary, which implies Theorem~\ref{t:simple}.

\begin{cor}\label{gh}
If $(u_k)$ is unbounded, the groups $S(u)^+$ and $S(u)$ admit no nontrivial continuous homomorphism into any compactly generated l.c.\ group.
\end{cor}
\begin{proof}
We have to check that the hypotheses of Proposition \ref{noho} are fulfilled. The topological simplicity is ensured by Proposition \ref{gsi}. The local condition also holds, because since $(u_k)$ is unbounded, every neighbourhood of the identity contains finite symmetric groups of all orders and thus contains elements of all possible finite orders.
\end{proof}

\subsection{Abstract homomorphisms of $S$ and $S^+$}\label{absh}
We start with the following converse to Corollary \ref{gh}.

\begin{prop}
If $(u_k)$ is bounded, then $S(u)^+$ and $S(u)$ are both embeddable as open subgroups in compactly generated l.c.\ groups, indeed topologically finitely generated.
\end{prop}
\begin{proof}
 Consider a permutation $\sigma$ of $\mathcal{G}$ preserving the partition by the height, with finitely many orbits. Then $\sigma$ normalizes $S$ and $S^+$, as well as $K$ and $K^+$. Therefore the semidirect products $S\rtimes\langle\sigma\rangle$ and $S^+\rtimes\langle\sigma\rangle$ are well-defined. They are totally disconnected locally compact groups containing $S$ (resp.\ $S^+$) as open subgroups. Moreover they act naturally by permutations of $\mathcal{G}$. The subgroup generated by $A^+$ and $\sigma$  is finitely generated (when $\sigma$ is transitive, this group was introduced by B.H. Neumann \cite{Neu}*{p.~127}). Since $A^+$ is dense in $S$ and $S^+$, it follows that $S \rtimes\langle\sigma\rangle $ and $S^+ \rtimes\langle\sigma\rangle$ are topologically finitely generated, hence compactly generated. 
\end{proof}

Using two theorems of S.~Thomas, it is possible to improve Corollary \ref{gh} in case the sequence $(u_k)$ tends to infinity.

\begin{thm}\label{remabs}
Assume that $\lim u_k=\infty$. Then $S(u)^+$ admits no nontrivial {\em abstract} homomorphism into any compactly generated l.c.\ group.
\end{thm}
\begin{proof}
We invoke the criterion from Proposition~\ref{homtri}, applied to  the group $S^+=S(u)^+$ endowed with the discrete topology. Thus, it is enough to show that:
\begin{enumerate}
\item\label{snc2} $S^+$ has no nontrivial action on any connected graph of bounded degree;
\item\label{snr2} $S^+$ has no nontrivial representation into $\GL_n(\CC)$.
\end{enumerate}
Both conditions can be checked with the help of the following result. Consider the subgroup $L_k=\prod_{j\ge k}\textnormal{Sym}(u_j)$ of $S^+$. Observe that $S^+$ is generated by the alternating finitary group ${A^+}$ and $L_k$ (because $A^+$ is dense), so it follows from  Proposition~\ref{gsi} that $S^+$ is normally generated by $L_k$. 

Next, we use a result of S.~Thomas \cite{Tho}*{Theorem 1.10} that every  (abstract) subgroup of at most countable index in $L_k$ is open. This immediately shows that every action of $S^+$ on a graph of at most countable valency is continuous, so (\ref{snc2}) follows from the proof of Corollary~\ref{gh} (which, through the proof of Proposition~\ref{noho}, discards the existence of a nontrivial continuous action on any connected graph of bounded valency).

Suppose $S^+$ has a non-trivial linear representation $\rho$ into some $\GL_d(\CC)$ over a field. Let $m_i$ be the dimension of the smallest nontrivial representation of the alternating group $\textnormal{Alt}(i)$; then $m_i$ tends to infinity (it can be shown that $m_i=i-1$ for $i\ge 7$, but a nice argument based on commutation \cite{Abe} gives a completely elementary lower bound $\simeq \sqrt{i}$). Fix $k$ so that $m_{u_j}>d$ for all $j\ge k$. Since $L_k$ normally generates $S^+$, the representation $\rho$ is non-trivial in restriction to $L_k$. By another result of S.~Thomas \cite{Tho}*{Theorem 2.1}, any non-trivial subgroup of $\GL_d(\CC)$ admits a subgroup of at most countable index. Apply this to $\rho(L_k)$ and let $H$ be its inverse image in $L_k$. By the choice of $k$, the kernel of $\rho$ contains the direct sum $\bigoplus_{j\ge k}\textnormal{Alt}(u_j)$, which is dense. So $H$ is dense; on the other hand the first-mentioned result of Thomas implies that $H$ is open. We thus reach a contradiction.
\end{proof}

We have seen in  Corollary~\ref{gh} that unboundedness of the sequence $(u_k)$ was sufficient to guarantee the absence of non-trivial homomorphisms of $S(u)$ or $S(u)^+$ to a compactly generated locally compact group. In contrast to this, the next result shows that  the hypothesis that $(u_k)$ tends to infinity in Theorem~\ref{remabs} cannot be weakened to the unboundedness of the sequence. 

\begin{prop}\label{qtg}
The quotient of the group $S$ (resp.\ $S^+$) by its subgroup of finitary permutations can be identified with 
\[\prod_{j}\textnormal{Sym}(u_j)/\bigoplus_j\textnormal{Sym}(u_j)\quad\textnormal{\Big(resp.}\quad \prod_{j}\textnormal{Alt}(u_j)/\bigoplus_j\textnormal{Alt}(u_j)\;\;\Big).\] In particular,
\begin{enumerate}
\item\label{gg} $S$ has an uncountable abstract abelianization and has   proper subgroups of finite index (such subgroups are necessarily dense). 
\item\label{gp} $S^+$ has proper subgroups of finite index if and only if $\liminf u_k<\infty$. It has a non-trivial  (resp.\ uncountable) abelianization if and only if $\liminf u_k\le 4$.
\end{enumerate}
\end{prop}
\begin{proof}
The first statement follows from the fact that $S=AK$, so $S/A=AK/A=K/(A\cap K)$; the argument for $S^+$ is similar. 

Denoting by $C_p$ the cyclic group of order $p$, the signature map induces a canonical surjection 
\[\prod_{j}\textnormal{Sym}(u_j)/\bigoplus_j\textnormal{Sym}(u_j)\to C_2^{\N}/C_2^{(\N)},\]
proving that $S$ has an uncountable abelianization, and, by taking suitable quotient, admits subgroups of index 2, proving (\ref{gg}). (Observe that $S^+$ has index 2 in the kernel of the surjection $S\to C_2^{\N}/C_2^{(\N)}$.) 

Concerning $S^+$, first assume that $\lim u_k=\infty$. If $S^+$ has nontrivial finite quotients, then it admits a nontrivial linear representation, contradicting Theorem~\ref{remabs}. Also observe that if for some $k$, $u_j\ge 5$ for $j\ge k$, then since $S^+$ is generated by the perfect groups $A^+$ and $L_k$, it is also perfect.

Conversely, assume $\liminf u_k<\infty$. 
Picking a subsequence on which $u_k$ is constant, say equal to $m\ge 3$, we obtain a surjective homomorphism 
$$S^+\to \textnormal{Alt}(m)^\N/\textnormal{Alt}(m)^{(\N)}.$$ 
Taking the limit with respect to a non-principal ultrafilter yields a nontrivial finite quotient. Also if $\liminf u_k\le 4$, then $S^+$ admits either $\textnormal{Alt}(3)^\N/\textnormal{Alt}(3)^{(\N)}$ or $\textnormal{Alt}(4)^\N/\textnormal{Alt}(4)^{(\N)}$ as a quotient, and thus admits either $C_2^{\N}/C_2^{(\N)}$ or $C_3^{\N}/C_3^{(\N)}$ as an uncountable abelian quotient.
\end{proof}

\section{Proof of Theorem \ref{opcl}}\label{qha}

Our first example is the following: let $\hat{\Q}$ be the Pontryagin dual of the discrete additive group $\Q$. So $\hat{\Q}$ is a connected, torsion-free compact group, and by Pontryagin duality, its automorphism group can be identified with the group of automorphisms of the group $\Q$, namely the multiplicative group $\Q^\times$. The first example is then
$$H_1=\hat{\Q}\rtimes\Lambda,$$ where $\Lambda$ is an arbitrary infinitely generated subgroup of $\Q^\times$ endowed with the discrete topology (recall that $\Q^\times$ is isomorphic to the product of its subgroup of order 2 and of a free abelian group of countable rank). 

Our second example is very similar in the construction. Fix a prime $p$. Recall that the group $\Z_p^\times$ is uncountable (it is known to be isomorphic to the product of a finite abelian group with $\Z_p$). Let $\Lambda$ be a countable, infinitely generated subgroup of $\Z_p^\times$, and endow $\Lambda$ with the discrete topology. Our second example is $$H_2=\Z_p\rtimes\Lambda.$$

A third example is $$H_3=\R\rtimes\Lambda,$$ where $\Lambda$ is a countable infinitely generated subgroup of $\R$; this is a Lie group. 

\begin{prop}
If an l.c.~group $G$ admits an isomorphic copy of $H_i$ ($i=1,2,3$) 
as an open subgroup, then it admits a discrete quotient which is an infinitely generated abelian group. In particular $G$ is not compactly generated. 
\end{prop}
\begin{proof}
Since the identity component $H^\circ=\hat{\Q}$ is open in $H$, it is open in $G$ and thus $H^\circ=G^\circ$ is open and normal in $G$. Thus the action by conjugation of $G$ on $G^\circ=\hat{\Q}$ defines a continuous homomorphism $\phi:G\to\Q^\times$, which is the identity on $\Lambda$ and trivial on $G^\circ$. So $\Ker(\phi)$ is open and the image of $\phi$ contains $\Lambda$ and is thus an infinitely generated abelian group. This concludes the proof of the proposition for $H_1$. The proof for $H_3$ is similar.

Let us now deal with $H_2$; since $\Z_p$ is not connected, the previous argument does not work. The subgroup $\Z_p$ being compact and open in $G$, it is commensurated by $G$; its abstract topological commensurator is the group $\Q_p^\times$, so that $G$ naturally admits a continuous homomorphism to $\Q_p^\times$, whose kernel contains $\Z_p$. Let us check it directly: observe that if $g\in G$, then there exists $n$ such that $g(p^n \Z_p)g^{-1}\subset\Z_p$, and then there exists a unique $\lambda(g)$, not depending on $n$, such that the conjugation by $g$, in restriction to $p^n\Z_p$, coincides with the multiplication by $\lambda(g)$. An immediate verification shows that $\lambda$ is a homomorphism. In restriction to $\Lambda$, the map $\lambda$ is the identity, and $\Ker(\lambda)$ is open in $G$ since it contains $\Z_p$. Thus $G/\Ker(\lambda)$ is a discrete abelian group containing $\Lambda$ and therefore fails to be finitely generated.
\end{proof}

In order to conclude the proof of Theorem~\ref{opcl}, it remains to show that those examples admit embeddings as closed subgroups into some compactly generated l.c.\ groups. 

For $H_2$, such an embedding can be obtained as follows. First embed $\Lambda$ into a finitely generated group $\Gamma$ (this is possible by Lemma~\ref{Zgamma}). Thus $\Lambda$ can be diagonally embedded as a discrete subgroup into $\Z_p^\times \times\Gamma$. This embedding extends to a continuous embedding of $\Z_p\rtimes\Lambda$ into $(\Z_p\rtimes\Z_p^\times)\times\Gamma$. This second embedding is continuous and injective; moreover it is proper since it is a discrete embedding in restriction to the cocompact subgroup $\Lambda$.

An obvious similar construction works for the third example $\R\rtimes\Lambda$. However, both embeddings rely on the fact that $\Lambda$ is contained in a compactly generated  l.c.~group of automorphisms of the normal subgroup ($\Z_p$ or $\R$). For $H_1$, we use the following topological version of a classical theorem of Krasner and Kaloujnine \cite{KK}.

\medskip
Recall that given two groups $K$ and $Q$, the {\bf unrestricted wreath product} $K \barwr  Q$ is the semidirect product $K^Q\rtimes Q$, where $Q$ acts on $K^Q$ by shifting on the left, namely $q\cdot f(r)=f(q^{-1}r)$. Assume now that $K$ is a topological group, and $Q$ is a discrete group. Then the product topology on $K^Q\times Q$ makes $K\barwr Q$ a topological group.

\begin{thm}\label{embwr}
For every l.c.~group $H$ that is an extension of a compact normal subgroup $K$ by a discrete quotient $Q$, there is an embedding of $H$ as a closed subgroup the unrestricted wreath product $K \barwr  Q=K^Q\rtimes Q$. 
\end{thm}

\begin{proof}[Proof of Proposition~\ref{p:compact-by-disc}]
Let $\Gamma$ be a finitely generated group containing $Q$. By Theorem~\ref{embwr}, there is a closed embedding $H  \leqslant K\barwr Q$. By the definition of the unrestricted wreath product, the embedding $Q  \leqslant \Gamma$ extends to a closed embedding $ K\barwr Q  \leqslant K\barwr \Gamma$.
\end{proof}

\begin{proof}[Proof of Theorem \ref{embwr}]
We begin by a general construction, not relying on the group topologies. Let $\pi:H\to Q$ be a surjective group homomorphism with kernel $K$. We will define, in a canonical way, a set $X=X(\pi)$ with commuting actions of $K\barwr Q$ and $H$, such that the $(K\barwr Q)$-action is simply transitive and the $H$-action is free. Given a choice of $x\in X$, this yields a unique injective homomorphism $F_x:G\to K\barwr Q$, mapping $g\in G$ to the unique element $s=F_x(g)\in K\barwr Q$ such that $gx=s^{-1}x$. The latter homomorphism depends on the choice of $x$, but is canonically defined up to post-composition by inner automorphisms of $K\barwr H$.

The set $X$ is defined to be the set of functions $f: Q\to H$ such that $\pi\circ f$ is a left translation of $Q$, by some element $\theta(f)$. Note that $X\neq\emptyset$, indeed it contains the set of set-theoretic sections $Q\to H$ of $\pi$, which are the elements $f$ in $X$ such that $\theta(f)=1$.

Let $K^Q$ act on $X$ as follows. If $u\in K^Q$, define $$u\cdot f(q)=f(q)u(q)^{-1}.$$ If $f\in X$ then $u\cdot f\in X$ and $\theta(u\cdot f)=\theta(f)$, because
$$\pi\circ(u\cdot f)(q)=\pi(f(q)u(q)^{-1})=\pi(f(q))=\theta(f)q.$$
This is clearly an action.

Besides, let $Q$ act on $X$ as follows. If $r\in Q$, define $$r\cdot f(q)=f(r^{-1}q).$$
Note that $\pi(r\cdot f(q))=\pi(f(r^{-1}q))=\theta(f)r^{-1}q$, so $r\cdot f\in X$ and $\theta(r\cdot f)=\theta(f)r^{-1}$. 

We next claim that these actions define an action of the semidirect product $K\barwr Q$ on $X$. To verify the claim, we need to show that for all $f \in X$, $u \in K^Q$  and $r \in Q$, we have
$$v \cdot f = r \cdot (u \cdot (r^{-1} \cdot f)),$$
where $v \in K^Q$ is defined as $v \colon  q \mapsto u(r^{-1} q)$. In other words we have $v= rur^{-1}  $ in the wreath product $K\barwr Q$. 
Given $q \in Q$, we have  
$$v\cdot f(q)=f(q)v(q)^{-1}=f(q)u(r^{-1}q)^{-1}. $$
On the other hand, we have
$$ \begin{array}{rcl}
r \cdot (u \cdot (r^{-1} \cdot f))(q) &= & (u \cdot (r^{-1} \cdot f))(r^{-1}q) \\
& = & (r^{-1} \cdot f)(r^{-1}q)u(r^{-1}q)^{-1} \\
& = &   f(r r^{-1}q)u(r^{-1}q)^{-1}\\
& = &  f(q)u(r^{-1}q)^{-1},
\end{array}$$
so that $v \cdot f(q) = r \cdot (u \cdot (r^{-1} \cdot f))(q)$ for all $q \in Q$, as desired. 

A straightforward verification shows that the action of $K\barwr Q$ on $X$ that has just been defined is simply transitive.

Finally, the $H$-action on $X$ is defined as follows: if $g\in H$ and $f$ is a function $Q\to H$, define $$g\cdot f(q)=gf(q).$$ If $f\in X$ and $g\in H$ and $q\in Q$, we have
$$(\pi\circ (g\circ f))(q)=\pi((g\circ f)(q))=\pi(gf(q))=\pi(g)\pi(f(q))=\pi(g)\theta(f)q, $$
so $g\cdot f\in X$ and $\theta(g\cdot f)=\pi(g)\theta(f)$.

We immediately see that the action of $H$, which is free, commutes with both the action of $K^Q$ and the action of $Q$, and thus commutes with the action of $K\barwr Q$.  So we have, for $x\in X$, an injective homomorphism $F_x:H\to K\barwr Q$ as defined above.

\medskip
Assume now that $K$ is a topological group, while $Q$ is still assumed to be discrete, so that $K\barwr Q$ is a topological group. Endow $H^Q$ with the product topology, and endow $X\subset H^Q$ with the topology induced by inclusion, namely the pointwise convergence topology. It is straightforward that the actions of $K\barwr Q$ and $H$ on $X$ are continuous and that orbital maps $K\barwr Q\to X$ are homeomorphisms. It follows that the homomorphism $F_x$ is continuous.

Let us now assume that $K$ is compact, so that $K\barwr Q$ and $X$ are both locally compact. (As soon as $Q$ is infinite, the converse holds, namely:  $K\barwr Q$ is locally compact if and only if $K$ is compact.) We claim that  the homomorphism $F_x$ is then proper. Checking this amounts to verify that the $H$-action on $X$ is proper.
Indeed, let $U_1,U_2$ be non-empty compact subsets of $X$ and let us check that $I=\{g\in H:gU_1\subset U_2\}$ has compact closure. By compactness, $\theta(U_2)$ is finite, and therefore we deduce that $\pi(I)$ is finite. Since $\pi$ is proper, it follows that $I$ has compact closure.
\end{proof}

\end{document}